\documentclass[a4paper,11pt,twoside]{amsart}

\usepackage{amsfonts,amsmath}
\usepackage{amssymb, latexsym}
\usepackage{amscd,amsthm}
\usepackage{mathrsfs}
\usepackage[all]{xy}
\usepackage[colorlinks=true,citecolor=blue, urlcolor=blue, linkcolor=blue]{hyperref}

\newcommand{\marginlabel}[1]%
 {\mbox{}\marginpar{\toggedleft\hspace{0pt}\bfseries\sf#1}}

\numberwithin{equation}{section}

\newtheorem{thm}{Theorem}[section]
\newtheorem{lem}[thm]{Lemma}
\newtheorem{prop}[thm]{Proposition}

\newtheorem{thmInt}{Theorem}[section]

\theoremstyle{definition}
\newtheorem{defn}[thm]{Definition}

\newtheorem{rem}[thm]{Remark}

\newcommand{\uQ}{{\widetilde Q}}
\newcommand{\uK}{{\widetilde K}}

\newcommand{\uY}{{\widetilde Y}}
\newcommand{\ul}{{\widetilde l}}
\newcommand{\uPP}{{\widetilde \PP}}

\newcommand{\A}{\mathcal{A}}

\newcommand{\B}{\mathcal{B}}

\newcommand{\CC}{\mathbb{C}}

\newcommand{\E}{\mathcal{E}}

\renewcommand{\H}{\mathcal{H}}
\newcommand{\I}{\mathcal{I}}

\newcommand{\M}{\mathcal{M}}

\renewcommand{\O}{\mathcal{O}}

\newcommand{\PP}{\mathbb{P}}

\newcommand{\QQ}{\mathbb{Q}}

\newcommand{\T}{\mathcal{T}}

\newcommand{\ZZ}{\mathbb{Z}}
\renewcommand{\geq}{\geqslant}
\renewcommand{\leq}{\leqslant}

\newcommand{\res}[2]{\left.#1\right|_{#2}} % restriction bar
\newcommand{\cat}[1]{\begin{bf}#1\end{bf}}

\newcommand{\lra}{\longrightarrow}

\newcommand{\set}[1]{\left\{#1\right\}}

\DeclareMathOperator{\coh}{\cat{Coh}}
\DeclareMathOperator{\Hom}{Hom}
\DeclareMathOperator{\Ext}{Ext}

\DeclareMathOperator{\id}{id}

\DeclareMathOperator{\rk}{rk}
\DeclareMathOperator{\RHom}{RHom}

\DeclareMathOperator{\Pic}{Pic}
\DeclareMathOperator{\ch}{ch}

\DeclareMathOperator{\codim}{codim}

\newcommand{\Db}{{\rm D}^{\rm b}}

\newcommand{\Mat}{\mathrm{End}}

\newenvironment{enumerate*}{\begin{enumerate}[topsep=4pt, partopsep=4pt, itemsep=0pt]}{\end{enumerate}}
\newenvironment{itemize*}{\begin{itemize}[topsep=4pt, partopsep=4pt, itemsep=0pt]}{\end{itemize}}
\newenvironment{description*}{\begin{description}[topsep=4pt, partopsep=4pt, itemsep=0pt]}{\end{description}}

\begin{document}

\title[ACM bundles on cubic fourfolds containing a plane]{Arithmetically Cohen-Macaulay bundles on cubic fourfolds containing a plane}

\author[M.~Lahoz, E.~Macr\`i, and P.~Stellari]{Mart\'{\i} Lahoz, Emanuele Macr\`i, and Paolo Stellari}

\address{M.L.: Institut de Math\'{e}matiques de Jussieu -- Paris Rive Gauche (UMR 7586), Universit\'{e} Paris Diderot / Universit\'{e} Pierre et Marie Curie, B\^{a}timent Sophie Germain, Case 7012, 75205 Paris Cedex 13, France}
\email{marti.lahoz@imj-prg.fr}
\urladdr{\url{http://webusers.imj-prg.fr/~marti.lahoz/}}

\address{E.M.: Department of Mathematics, The Ohio State University, 231 W 18th Avenue, Columbus, OH 43210, USA}
\curraddr{Department of Mathematics, Northeastern University, 360 Huntington Avenue, Boston, MA 02115, USA}
\email{e.macri@neu.edu}
\urladdr{\url{http://nuweb15.neu.edu/emacri/}}

\address{P.S.: Dipartimento di Matematica ``F.~Enriques'', Universit{\`a} degli Studi di Milano, Via Cesare Saldini 50, 20133 Milano, Italy}
\email{paolo.stellari@unimi.it}
\urladdr{\url{http://users.unimi.it/stellari}}

\thanks{M.~L.~ is partially supported by SFB/TR 45, Fondation Math\'ematique Jacques Hadamard (FMJH) and MTM2012-38122-C03-02.
E.~M.~ is partially supported by the NSF grants DMS-1001482/DMS-1160466 and DMS-1302730/DMS-1523496, the Hausdorff Center for Mathematics, Universit\"at Bonn, and SFB/TR 45. P.~S.~ is partially supported by the grants FIRB 2012 ``Moduli Spaces and Their Applications'' and
the national research project ``Geometria delle Variet\`a Proiettive'' (PRIN 2010-11).}

\keywords{Arithmetically Cohen-Macaulay vector bundles, cubic fourfolds}

\subjclass[2010]{18E30, 14E05}

\begin{abstract}
	We study ACM bundles on cubic fourfolds containing a plane exploiting the geometry of the associated quadric fibration and Kuznetsov's treatment of their bounded derived categories of coherent sheaves.
	More precisely, we recover the K3 surface naturally associated to the fourfold as a moduli space of Gieseker stable ACM bundles of rank four.
\end{abstract}

\maketitle

%%%%%%%%%%%%%%%%%%%%%%%%%%%%%%%%%%%%%

\section*{Introduction}

This paper is a follow-up of \cite{LMS3} which deals with some geometric properties of special ACM bundles on smooth cubic threefolds (i.e., smooth hypersurfaces of degree $3$ in $\PP^4$). The idea we are pursuing in the two papers is that, for a cubic hypersurface $Y$, one should consider semiorthogonal decompositions of its derived category $\Db(Y)$.
Up to some (a priori non-canonical) choice, one gets a non-trivial triangulated subcategory $\mathbf{T}_Y\subset\Db(Y)$.
According to an intuition of Kuznetsov, $\mathbf{T}_Y$ should encode the birational information of the cubic.

More specifically, if $Y$ is a cubic fourfold (i.e., a smooth hypersurface of degree $3$ in $\PP^5$), then in \cite{Kuz:4fold} it is conjectured that $Y$ is rational if and only if the category $\mathbf{T}_Y$ is equivalent to the derived category of a K3 surface. A relation between this conjecture and the classical Hodge theoretical approach to rationality appears in \cite{AT}.

\smallskip

If we assume further that the cubic fourfold $Y$ contains a plane $P$, then the projection from $P$ onto a skew $\PP^2$ in $\PP^5$ yields a quadric fibration $\pi$ over $\PP^2$. As an instance of Kuznetsov's semiorthogonal decomposition of the derived category of a quadric fibration (see \cite{Kuz:Quadric} and Section \ref{subsec:4foldsgeom}), there exists an exact equivalence $\Xi$ between $\mathbf{T}_Y$ and the bounded derived category of sheaves on $\PP^2$ with the action of a sheaf of Clifford algebras $\B_0$ (determined by the structure of quadric fibration on the cubic).

The category $\Db(\PP^2,\B_0)$ can be described more geometrically.
Indeed, the singular quadrics of the fibration $\pi$ lie over a sextic plane curve which we denote by $C$.
Generically, the sextic is smooth.
In such a case, we let $S$ be the smooth projective K3 surface obtained as double cover of $\PP^2$ ramified along $C$.
We call $S$ the associated K3 surface to the pair $(Y,P)$ or, equivalently, to the quadric fibration $\pi$.
Otherwise, the K3 surface $S$ is singular over the singular points of $C$.
Then, by \cite{Kuz:4fold}, we have an equivalence $\Db(\PP^2,\B_0)\cong\Db(S,\A_0)$, where $\A_0$ is a sheaf of algebras over $S$ which, restricted to the smooth locus of $S$, is actually a sheaf of Azumaya algebras.

\subsection*{The result}

Consider now a very ample line bundle $\O_X(H)$ on a smooth projective variety $X$.
Recall that a vector bundle $F$ on a $X$ is \emph{Arithmetically Cohen--Macaulay} if $\dim H^i(X,F(jH))=0$, for all $i=1,\ldots,n-1$ and all $j\in\ZZ$.
The presence of families of arbitrary dimensions of such bundles determines the \emph{representation type} of $X$ which should encode the complexity of the geometry of $X$ (see, for example, \cite{CH1}).
A way to make this precise is by observing that ACM bundles correspond to {\em Maximal Cohen-Macaulay} (MCM) modules over the graded ring associated to the projectively embedded variety (see, e.g., \cite{Yo}).

In \cite{LMS3}, it was observed that, given a stable ACM bundle $F$ on a smooth projective cubic hypersurface $Y$, a certain twist of $F$ by the very ample line bundle $\O_Y(H)$ belongs to $\mathbf{T}_Y$ (this is Lemma \ref{lem:ACM}).
If $Y$ has dimension $4$ and contains a plane, one can use the functor $\Xi$ to study basic properties of ACM bundles on $Y$ or to construct examples or families of such bundles.
This reduces the problem to consider certain complexes of $\B_0$-modules on $\PP^2$.
From the homological point of view, $\Db(\PP^2,\B_0)$ has dimension $2$ and thus one may reasonably hope that such a dimension reduction can clarify or simplify the picture and the computations.

\smallskip

The present paper is actually an incarnation of this idea.
More precisely, the K3 surface $S$ associated to $(Y,P)$ plays an important role in the study of moduli spaces of ACM bundles on cubic fourfolds.
Indeed, we can prove the following.

\begin{thmInt}\label{thm:main4folds}
	Let $Y$ be a cubic fourfold in $\PP^5$ containing a plane $P$.
	Then the smooth locus $S_{\mathrm{reg}}$ of the K3 surface $S$ associated to $(Y,P)$ is isomorphic to an open subset of an irreducible component of the moduli space of Gieseker stable ACM bundles over $Y$ with class $\left(4,-2H,-P,l,\tfrac{1}{4}\right)$, where $l$ is the class of a line in $Y$.
\end{thmInt}

The theorem can be deduced by a quite long computation carried out all along Sections \ref{sec:family1} and \ref{sec:family2}.
As far as we know, this is the first example of a $2$-dimensional family of stable ACM bundles of rank $4$ on a cubic fourfold.
Moreover, we observe that the smooth locus of any moduli space of slope-stable ACM vector bundles on $Y$ carries a symplectic form
(see Remark \ref{rem:sympl}).

A way to rephrase Theorem \ref{thm:main4folds} is that the embedding of $\cat{T}_Y$ into $\Db(Y)$ can be realized as a fully faithful Fourier--Mukai functor whose kernel is a shift of a sheaf: namely the universal family in the moduli problem in Theorem \ref{thm:main4folds}.

\subsection*{Related works}

As we pointed out in \cite{LMS3}, the idea of using semiorthogonal decompositions to study ACM bundles by reducing dimension is an application of the circle of ideas in \cite{KuzNew}.
This idea was previously exploited in \cite{MS} to describe the Fano variety of lines of cubic fourfolds containing a plane.
In \cite{BMMS}, the Fano variety of lines on a cubic threefold is reconstructed as a moduli space of Bridgeland stable objects.
Several ideas from the last two papers enter the picture described by the present paper and \cite{LMS3}.

Several papers study stable ACM bundles on threefolds and surfaces (see  \cite{LMS3} for a non-complete list of references).
In \cite{LMS3}, we provide a generalization of one of the main results in \cite{CH}.
In particular, we show that the moduli spaces of stable Ulrich bundles on any smooth cubic threefold is non-empty.
Roughly speaking, Ulrich bundles are ACM bundles with prescribed constraints on the degree of the generators (see \cite{CH} for a precise definition).

We should point out that the use of derived categories and Bridgeland stability is actually intrinsic in the strategy of the proofs of the main results in \cite{LMS3}.
On the other hand, the only key point where derived categories appear in the present paper is the existence of the Fourier--Mukai equivalence $\Xi$ above.
This serves as a guideline to construct the $2$-dimensional family of stable ACM bundles in Theorem \ref{thm:main4folds}.

\subsection*{Plan of the paper}

The paper is organized as follows.
Section \ref{sec:prelim} collects basic facts about semiorthogonal decompositions and general results about ACM bundles on cubic hypersurfaces.
In Section \ref{subsec:4foldsgeom} we review Kuznetsov's work on quadric fibrations and we then focus on the case of cubic fourfolds containing a plane.

Sections \ref{sec:family1} and \ref{sec:family2} are devoted to the proof of Theorem \ref{thm:main4folds}. In particular, in Section \ref{sec:family1}, we show how to associate a (shift of a) coherent sheaf on $Y$ to any point in the regular locus of the K3 surface $S$ associated to $Y$. We then prove in Section \ref{sec:family2} that such sheaves are Gieseker stable and of ACM type.

\subsection*{Notation}

Throughout this paper we work over the complex numbers.
For a smooth projective variety $X$, we denote by $\Db(X)$ the bounded derived category of coherent sheaves on $X$.
We refer to \cite{huy} for basics on derived categories.
If $X$ is not smooth, we denote by $X_{\mathrm{reg}}$ the regular part of $X$.
This paper assumes some familiarity with the notion of slope and Gieseker stability, of Harder--Narasimhan (HN) and Jordan--H\"{o}lder (JH) factors of a (semistable) vector bundle.
For this, we refer to \cite{HL} which is also our reference for the standard construction of moduli spaces of stable sheaves. To shorten the notation Gieseker stability will be simply called stability.

\section{Preliminaries}\label{sec:prelim}

This section contains some preliminary material about semiorthogonal decompositions and their use to study quadric fibrations. We then specialize to the case of cubic fourfolds containing a plane with particular attention to the associated K3 surface.

\subsection{Semiorthogonal decompositions}\label{subsec:Semiorth}
Let $X$ be a a smooth projective variety and let $\Db(X)$ be its bounded derived category of coherent sheaves. A \emph{semiorthogonal} decomposition of $\Db(X)$ is a sequence of full triangulated subcategories $\cat{T}_1,\ldots,\cat{T}_m\subseteq\Db(X)$ such that $\Hom_{\Db(X)}(\cat{T}_i,\cat{T}_j)=0$, for $i>j$ and, for all $G\in\Db(X)$, there exists a chain of morphisms in $\Db(X)$
	\[
	0=G_m\to G_{m-1}\to\ldots\to G_1\to G_0=G
	\]
	with $\mathrm{cone}(G_i\to G_{i-1})\in\cat{T}_i$, for all $i=1,\ldots,m$.
	We will denote such a decomposition by $\Db(X)=\langle\cat{T}_1,\ldots,\cat{T}_m\rangle$.

Moreover, an object $F\in\Db(X)$ is \emph{exceptional} if $\Hom_{\Db(X)}(F,F)\cong\CC$ and $\Hom_{\Db(X)}^p(F,F)=0$, for all $p\neq0$.
	A collection $\{F_1,\ldots,F_m\}$ of objects in $\Db(X)$ is called an \emph{exceptional collection} if $F_i$ is an exceptional object, for all $i$, and $\Hom_{\Db(X)}^p(F_i,F_j)=0$, for all $p$ and all $i>j$.

\begin{rem}\label{rmk:exceptional}
	An exceptional collection $\{F_1,\ldots,F_m\}$ in $\Db(X)$ provides a semiorthogonal decomposition
	\[
	\Db(X)=\langle\cat{T},F_1,\ldots,F_m\rangle,
	\]
	where, by abuse of notation, we denoted by $F_i$ the triangulated subcategory generated by $F_i$ (equivalent to the bounded derived category of finite dimensional vector spaces).
	Moreover
	\[
	\cat{T}:=\langle F_1,\ldots,F_m\rangle^\perp=\left\{G\in\Db(X)\,:\,\Hom^p(F_i,G)=0,\text{ for all }p\text{ and }i\right\}.
	\]
	Similarly, one can define ${}^\perp\langle F_1,\ldots,F_m\rangle=\left\{G\in\cat{T}\,:\,\Hom^p(G,F_i)=0,\text{ for all }p\text{ and }i\right\}$.
\end{rem}

For $F\in \Db(X)$ an exceptional object, we consider the two functors, respectively \emph{left and right mutation}, $\cat{L}_F,\cat{R}_F:\Db(X)\to\Db(X)$ defined by
\begin{equation}\label{eqn:LRmutation}
	\begin{split}
	\cat{L}_F(G)&:=\mathrm{cone}\left(\mathrm{ev}:\mathrm{RHom}(F,G)\otimes F\to G\right)\\
	\cat{R}_F(G)&:=\mathrm{cone}\left(\mathrm{ev}^\vee:G\to\mathrm{RHom}(G,F)^\vee\otimes F\right)[-1],
	\end{split}
\end{equation}
where $\mathrm{RHom}(-,-):=\oplus_{p}\Hom_{\Db(X)}^p(-,-)[-p]$.
More intrinsically, let $\iota_{{}^\perp F}$ and $\iota_{F^\perp}$ be the full embeddings of ${}^\perp F$ and $F^\perp$ into $\Db(X)$.
Denote by $\iota^*_{{}^\perp F}$ and $\iota^!_{{}^\perp F}$ the left and right adjoints of $\iota_{{}^\perp F}$ and by $\iota^*_{{F}^\perp}$ and $\iota^!_{{F}^\perp}$ the left and right adjoints of $\iota_{{F}^\perp}$.
Then $\cat{L}_F=\iota_{F^\perp}\circ\iota^*_{F^\perp}$, while $\cat{R}_F=\iota_{{}^\perp F}\circ\iota^!_{{}^\perp F}$ (see, e.g., \cite[Sect.~2]{KuzHPD}).

The main property of mutations is that, given a semiorthogonal decomposition of $\Db(X)$
	\[
	\langle\cat{T}_1,\ldots,\cat{T}_k,F,\cat{T}_{k+1},\ldots,\cat{T}_n\rangle,
	\]
we can produce two new semiorthogonal decompositions
	\[
	\langle\cat{T}_1,\ldots,\cat{T}_k,\cat{L}_F(\cat{T}_{k+1}),F,\cat{T}_{k+2},\ldots,\cat{T}_n\rangle
	\]
and regarding ACM bundles on threefolds and surfaces.
	\[
	\langle\cat{T}_1,\ldots,\cat{T}_{k-1},F,\cat{R}_F(\cat{T}_k),\cat{T}_{k+1},\ldots,\cat{T}_n\rangle.
	\]

\subsection{ACM bundles on cubics}\label{subsec:ACM}

Let us briefly summarize some general results form \cite{LMS3}. Let $Y$ be a \emph{smooth cubic $n$-fold}, namely a smooth projective hypersurface of degree $3$ in $\PP^{n+1}$.
We set $\O_Y(H):=\O_{\PP^{n+1}}(H)|_{Y}$.
According to Remark \ref{rmk:exceptional}, as observed by Kuznetsov, the derived category $\Db(Y)$ of coherent sheaves on $Y$ has a semiorthogonal decomposition
\begin{equation}\label{eqn:semiorthgen}
		\Db(Y)=\langle\cat{T}_Y,\O_Y,\O_Y(H),\ldots,\O_Y((n-2)H)\rangle,
\end{equation}
where, by definition,
\begin{equation*}
	\begin{split}
	\cat{T}_Y&:=\langle\O_Y,\ldots,\O_Y(n-2)\rangle^\perp\\
	&=\left\{G\in\Db(Y):\Hom^p_{\Db(Y)}(\O_Y(iH),G)=0,\text{ for all }p\text{ and }i=0,\ldots,n-2\right\}.
	\end{split}
\end{equation*}

To begin with, consider the following. regarding ACM bundles on threefolds and surfaces.

\begin{defn}\label{def:ACMbalanced}
	(i) A vector bundle $F$ on a smooth projective variety $X$ of dimension $n$ is \emph{arithmetically Cohen-Macaulay} (ACM) if $\dim H^i(X,F(jH))=0$, for all $i=1,\ldots,n-1$ and all $j\in\ZZ$.
	
	(ii) An ACM bundle $F$ is called \emph{balanced} if $\mu(F)\in[-1,0)$.
\end{defn}

The following results will be relevant in the rest of the paper and provide even more evidence of the geometric meaning of the admissible subcategory $\mathbf{T}_Y$.

\begin{lem}\label{lem:ACM}{\bf(\cite{LMS3}, Lemma 1.6.)}
	Let $Y\subset \PP^{n+1}$ be a smooth cubic $n$-fold.
	Let $F$ be a balanced $\mu$-stable ACM bundle with $\rk(F)>1$.
	Then $F\in \mathbf{T}_Y$.
\end{lem}

It was observed in \cite{LMS3}, that the lemma above is slightly more general and the same proof works for a balanced ACM bundle of rank greater than one, if it is $\mu$-semistable and $\Hom(F,\O_Y(-H))=0$.

\begin{rem}\label{rem:sympl}
When $n=4$, the Serre functor of the subcategory $\cat{T}_Y$ is isomorphic to the shift by $2$ (see \cite[Thm.~4.3]{Kuz:4fold}).
Thus, as an application of the result above and \cite[Thm.~4.3]{KM}, one gets that the smooth locus of any moduli space of $\mu$-stable ACM vector bundles on $Y$ carries a closed symplectic form.
\end{rem}

\begin{lem}\label{lem:viceversa}{\bf(\cite{LMS3}, Lemma 1.8.)}
	Let $Y\subset \PP^{n+1}$ be a smooth cubic $n$-fold and let $F\in \coh(Y) \cap \cat{T}_Y$.
	Assume
	\begin{equation}\label{eqn:vanishing1}
		\begin{split}
		& H^1(Y,F(H))=0\\
		& H^1(Y,F((1-n)H))=\ldots=H^{n-1}(Y,F((1-n)H))=0.
		\end{split}
	\end{equation}
	Then $F$ is an ACM bundle.
\end{lem}

\subsection{Geometry of cubic fourfolds with a plane}\label{subsec:4foldsgeom}

In this section, we let $Y\subset\PP^5$ be a cubic fourfold containing a plane $P$.
Consider the blow-up $\uPP$ of $\PP^5$ along $P$.
We set $q: \uPP \to \PP^2$ to be the $\PP^3$-bundle induced by the projection from $P$ onto a plane and we denote by $\uY$ the strict transform of $Y$ via this blow-up.
The restriction of $q$ to $\uY$ induces a quadric fibration $\pi: \uY \to \PP^2$.
Note that $\uPP=\PP_{\PP^2}(E)$, where $E$ is the vector bundle $\O_{\PP^2}^{\oplus 3} \oplus \O_{\PP^2}(-h)$ on $\PP^2$.

The fibres of $\pi$ degenerate along a sextic $C\subset \PP^2$.
The curve $C$ has at most ordinary double points.
On the one hand, over the smooth points of $C$ the fibres are cones with one singular point.
On the other hand, over the singular points of $C$ the fibre is the union of two planes intersecting along a line.
For the general cubic fourfold containing a plane, the sextic $C$ is smooth.
The double cover over $f\colon S\to\PP^2$ ramified along $C$ is a projective K3 surface (singular over the singular points of $C$).
The geometric picture can be summarized by the following diagram
\begin{equation}\label{eqn:4fold}
	\xymatrix{
	D\ar[d]_-{s}\ar@{^{(}->}[r]&\uY \ar@{^{(}->}[rr]^-{\alpha}\ar[d]_-{\sigma}\ar[rrd]_-{\pi}&&
	\uPP=\PP_{\PP^2}(\O_{\PP^2}^{\oplus 3}\oplus \O_{\PP^2}(-h))\ar[d]_-{q}& \\
	P\ar@{^{(}->}[r]& Y\subset \PP^5 && C\subset\PP^2& S\ar[l]^-{f}.
	}
\end{equation}
We let $D \subset \uY$ be the exceptional divisor of the blow-up $\sigma: \uY \to Y$.
We denote by $h$ both the class of a line in $\PP^2$ and its pull-backs to $\uPP$ and $\uY$ and, accordingly, we call $H$ both the class of a hyperplane in $\PP^5$ and its pull-backs to $Y$, $\uPP$, and $\uY$.
We recall that $\O_{\uY}(D)\cong\O_{\uY}(H-h)$, the relative ample line bundle is $\O_{\uPP}(H)$, the relative canonical bundle is $\O_{\uPP}(h-2H)$, and the dualizing sheaf of $\uY$ is $\omega_{\uY}\cong \O_{\uY}(-h-2H)$ (see, e.g., \cite[Lem.~4.1]{Kuz:4fold}).

\smallskip

According to \cite{Kuz:Quadric}, the quadric fibration $\pi:\uY\to\PP^2$ carries a \emph{sheaf $\B$ of Clifford algebras} which is the relative sheafified version of the classical Clifford algebra associated to a quadric on a vector space (more details can be found in \cite[Sect.\ 3]{Kuz:Quadric}).
As in the absolute case, $\B$ has an \emph{even part} $\B_0$ whose description as an $\O_{\PP^2}$-module is as follows
\[
\B_0\cong\O_{\PP^2}\oplus(\wedge^2 E\otimes L)\oplus(\wedge^4 E\otimes L^2)\oplus\ldots,
\]
where $L:=\O_{\PP^2}(-h)$. The \emph{odd part} $\B_1$ of $\B$ is such that
\[
\B_1\cong E\oplus(\wedge^3 E\otimes L)\oplus(\wedge^5 E\otimes L^2)\oplus\ldots
\]
We also denote $\B_{2i}=\B_0\otimes L^{-i}$ and $\B_{2i+1}=\B_1\otimes L^{-i}$. As sheaves of $\O_{\PP^2}$-modules, we have the following isomorphisms
\begin{equation}\label{bcoh1}
	\begin{split}
	\B_0&\cong\O_{\PP^2}\oplus\O_{\PP^2}(-h)^{\oplus 3}\oplus\O_{\PP^2}(-2h)^{\oplus 3}\oplus\O_{\PP^2}(-3h),\\
	\B_1&\cong\O_{\PP^2}^{\oplus 3}\oplus\O_{\PP^2}(-h)^{\oplus 2}\oplus\O_{\PP^2}(-2h)^{\oplus 3}.
	\end{split}
\end{equation}

The category $\coh(\PP^2,\B_0)$ is the abelian category of coherent $\B_0$-modules on $\PP^2$ and $\Db(\PP^2,\B_0)$ is its derived category.

\smallskip

As explained in \cite{Kuz:Quadric}, there exists a fully faithful functor $\Phi:=\Phi_{\E'}:\Db(\PP^2,\B_0)\to\Db(\uY)$ is defined as the Fourier--Mukai transform
\begin{equation*}\label{eqn:defemb}
	\Phi_{\E'}(-):=\pi^*(-)\otimes_{\pi^*\B_0}\E',
\end{equation*}
where $\E'\in\coh(\uY)$ is a rank $4$ vector bundle on $\uY$ with a natural structure of flat left $\pi^*\B_0$-module defined by the short exact sequence
\begin{equation}\label{eqn:defE'}
	0\longrightarrow q^*\B_0(-2H)\longrightarrow q^*\B_1(-H)\longrightarrow\alpha_*\E'\longrightarrow
0.
\end{equation}
Such an exact functor has a left adjoint
\begin{equation}\label{eqn:bascoh}
	\Psi(-):=\pi_*((-)\otimes_{\O_\uY}\E\otimes_{\O_\uY}\O_{\uY}(h)[1]),\\
\end{equation}
where $\E\in\coh(\uY)$ is another rank $2$ vector bundle with a natural structure of right $\pi^*\B_0$-module (see again \cite[Sect.\ 4]{Kuz:Quadric}).
The analogous presentation of $\E$ is
\begin{equation}\label{eqn:defE}
	0\longrightarrow q^*\B_{1}(-h-2H)\longrightarrow q^*\B_0(-H)\longrightarrow\alpha_*\E\longrightarrow
0.
\end{equation}

\smallskip

In \cite[Thm.~4.3]{Kuz:4fold}, Kuznetsov constructs an equivalence $\Db(\PP^2,\B_0)\cong\cat{T}_Y$, where $\cat{T}_Y$ is the full subcategory in \eqref{eqn:semiorthgen}.
The way this equivalence is obtained is by performing a precise sequence of mutations which allow Kuznetsov to compare the semiorthogonal decomposition of the derived category of $\uY$ in \cite[Thm.~4.2]{Kuz:Quadric} and the one
\begin{equation*}
	\Db(\uY)=\langle\sigma^*(\cat{T}_Y),\O_{\uY},\O_{\uY}(H),\O_{\uY}(2H),i_*\O_{D},i_*\O_{D}(H),i_*\O_{D}(2H)\rangle
\end{equation*}
obtained by thinking of $\uY$ as the blow-up of $Y$ along $P$ and using \cite{Orlov}.
The details will not be needed in the rest of this paper but we just recall that
\begin{equation*}
	\cat{T}_Y = \sigma_*\circ\cat{L}_{\O_{\uY}(h-H)}\circ \cat{R}_{\O_{\uY}(-h)} \circ \Phi (\Db(\PP^2,\B_0)).
\end{equation*}
For later use we set
\begin{equation}\label{eqn:Xi4fold}
	 \Xi:=(\sigma_*\circ\cat{L}_{\O_{\uY}(h-H)}\circ\cat{R}_{\O_{\uY}(-h)}\circ\Phi)^{-1}:\cat{T}_Y\to \Db(\PP^2,\B_0).
\end{equation}

\subsection{The derived category of the associated K3 surface}\label{subsec:K3}

Let $S$ be the K3 surface associated to a cubic fourfold containing a plane, as in \eqref{eqn:4fold}.
By \cite[Sect.\ 3.5]{Kuz:Quadric}, there exists a sheaf of algebras $\A_0$ such that $f_*(\A_0)=\B_0$ and $f_*:\coh(S,\A_0) \to \coh(\PP^2,\B_0)$ is an equivalence.
Moreover, by \cite[Prop.~3.13]{Kuz:Quadric}, $\A_0$ restricted to the smooth locus $S_{\mathrm{reg}}$ of $S$ is a sheaf of Azumaya algebras.
For the basic properties of Azumaya algebras, see \cite[Ch. IV]{Milne} or \cite[Chapter 1]{C}.

When $C$ is smooth, we can describe the category $\cat{Coh}(S,\A_0)$ in terms of twisted sheaves.
More precisely, there exist $\alpha \in \mathrm{Br}(S)$ in the \emph{Brauer group} of $S$, $\alpha^2=\id$, and an $\alpha$-twisted vector bundle of rank 2, $E_\alpha\in \coh(S,\alpha)$, such that $\A_0=\mathcal{E}nd(E_{\alpha})$ and
\begin{equation*}
	 \left.\begin{array}{rcl}
	 {\coh(S,\alpha)}&\stackrel{\sim}\lra &{\coh(S,\A_0)}\\{F}&\longmapsto& {F\otimes E_\alpha^\vee=\mathcal{H}om(E_{\alpha},F)}
	 \end{array}\right.
\end{equation*}
is an equivalence of categories.
When $C$ is singular, the vector bundle $E_\alpha$ still exists \'etale locally on smooth points.

Let $x\in S_{\mathrm{reg}}$.
Consider $L_x:=f_*(\CC(x)\otimes E_\alpha^\vee)\in \coh(\PP^2,\B_0)$.
As an $\O_{\PP^2}$-module it is isomorphic to $V\otimes_\CC \CC(f(x))$, where $V$ is a $2$-dimensional $\CC$-vector space. The structure of $\B_0$-module on $L_x$ is provided by the following result:
\begin{lem}\label{lem:B0act}
	\begin{itemize}
		\item[{\rm(i)}] If $f(x)\not\in C$, then $\res{\B_0}{f(x)}\cong\Mat(V)\times\Mat(V)$ and it acts on $V$ via one of the two projections $\Mat(V)\times\Mat(V)\to\Mat(V)$.
		\item[{\rm(ii)}] If $f(x)\in C_{\mathrm{reg}}$, then $\res{\B_0}{f(x)}\cong\Mat(V)$ and it acts on $V$ via the standard representation.
	\end{itemize}
\end{lem}

\begin{proof}
	This follows directly from \cite[Lem.~2.6 and Prop.~3.13]{Kuz:Quadric}.
	% When $f(x)\not\in C$, then $E_{\alpha}^{\vee}\otimes\CC(x)$ is isomorphic to $V$ and $({\A_0})_x$ is isomorphic to $\Mat(V)$.
	% In this case, $\res{\B_0}{f(x)}\cong\Mat(V)\times\Mat(V)$ (see \cite[Sect.\ 2.4]{Kuz:Quadric}) and, clearly, $(\B_0)_{f(x)}$ acts through one of the projections.
	% When $f(x)\in C_{\mathrm{reg}}$, the fibre of the algebra $\A_0$ at $x$ isomorphic to $\Mat(V)$ (see \cite[Lem.~2.6 and Prop.~3.13]{Kuz:Quadric}) and, clearly also $\res{\B_0}{f(x)}\cong\Mat(V)$, which acts naturally on the vector space $V$.
\end{proof}

\section{A family of sheaves on cubic fourfolds containing a plane}\label{sec:family1}

As a first step in the proof of Theorem \ref{thm:main4folds}, in this section we want to describe a family of bundles parametrized by the K3 surface $S$ described in Section \ref{subsec:K3}.
More precisely, for $x\in S_{\mathrm{reg}}$ and setting
\begin{equation}\label{eqn:Mx}
	M_x:=\Xi^{-1}(L_x)[-1] \in \Db(Y)
\end{equation}
we get the following.

\begin{prop}\label{prop:family}
	For all $x\in S_{\mathrm{reg}}$, the object $M_x$ is a coherent sheaf with class
	 \[
	 \mathrm{ch}(M_x)=(4,-2H,-P,l,\tfrac{1}{4})\in H^*(Y,\QQ).
	 \]
\end{prop}

The proof will be carried out in the rest of this section and it is divided up in several steps.
Moreover, in Proposition \ref{prop:family2}, we will show that $M_x$ is actually a stable ACM bundle.

\smallskip

We denote by $\uQ_{f(x)}$ the fibre of $\pi\colon\uY\to\PP^2$ over $f(x)$ and by $\ul_x\subset \uQ_{f(x)}$ any line in the ruling corresponding to $x\in S_{\mathrm{reg}}$.
Recall that the points in $S_{\mathrm{reg}}$ parametrize rulings of lines in the quadric fibration $\pi$.
When it is clear from the context, we will denote $\uQ_{f(x)}$ simply by $\uQ$ and $\ul_x$ by $l$.

\subsection*{Step 1: Kuznetsov's embedding}

We first want to prove that
\begin{equation}\label{eqn:Kuz}
	\Phi(L_x)\cong \I_{\ul_x,\uQ_{f(x)}},
\end{equation}
where $\I_{\ul_x,\uQ_{f(x)}}$ is the ideal sheaf of $\ul_x$ in $\uQ_{f(x)}$.

Indeed, by definition $\Phi(L_x)= \pi^*L_x \otimes_{\pi^*\B_0} \E' \in \Db(\uY)$.
Since $\pi$ is flat and $\E'$ is $\pi^*\B_0$-flat, we have $\Phi(L_x)\in \coh(\uY)$.
As $\alpha$ is a closed immersion, also $\alpha_* \Phi(L_x)$ is a sheaf, i.e., $\alpha_* \Phi(L_x) \in \coh(\uPP)$.
We have
\[
\begin{split}
 \alpha_* \Phi(L_x) &\cong \alpha_*(\pi^*L_x \otimes_{\pi^*\B_0} \E')\\
&\cong q^*L_x \otimes_{q^*\B_0} \alpha_*\E'.
\end{split}
\]
where we used the Projection Formula for the fist isomorphism and the identity $\pi=q\circ\alpha$ for the second one.
Since $\alpha_* \Phi(L_x)$ is a sheaf, if we tensor the exact sequence \eqref{eqn:defE'} by $q^*L_x$, then we get an exact sequence:
\begin{equation}\label{eqn:compLx}
	0\to q^*L_x \otimes_{q^*\B_0}q^*{\mathcal B}_0(-2H)\stackrel{\delta}\to q^*L_x \otimes_{q^*\B_0}q^*{\mathcal B}_1(-H)\to q^*L_x
	\otimes_{q^*\B_0}\alpha_*{\mathcal E}'\to0.
\end{equation}

The first term in the previous exact sequence is
\begin{align*}
	 q^*L_x \otimes_{q^*\B_0}q^*{\mathcal B}_0(-2H)&\cong (q^*L_x \otimes_{q^*\B_0}q^*{\mathcal
	B}_0)\otimes_{\O_{\uPP}}\O_{\uPP}(-2H) \\
	&=q^*L_x \otimes_{\O_{\uPP}}\O_{\uPP}(-2H)\\
	&=\O_{\PP^3}(-2)^{\oplus 2},
\end{align*}
where in the first isomorphism we have used the non-commutative associativity (see, for example, {\cite[Prop.~A2.1]{eisenbud}}) and the fact that $q^*{\mathcal
B}_0(-2H)=q^*{\mathcal B}_0\otimes_{\O_{\uPP}}\O_{\uPP}(-2H)$.

The action of $\B_0$ on $L_x$ is controlled by Lemma \ref{lem:B0act}.
The map $\delta$ in \eqref{eqn:compLx} corresponds to $\delta_{-1,1}$ in \cite[Sect.~3.4]{Kuz:Quadric}.
Hence, the second term of \eqref{eqn:compLx} is $q^*L_x \otimes_{q^*\B_0}q^*\B_1(-H)\cong \O_{\PP^3}(-1)^{\oplus 2}$ and
\begin{equation}\label{eqn:matrixf}
	q^*L_x \otimes_{q^*\B_0}q^*{\mathcal B}_0(-2H)\cong \O_{\PP^3}(-2)^{\oplus 2}
	\stackrel{\delta}\to \O_{\PP^3}(-1)^{\oplus 2} \cong
	q^*L_x \otimes_{q^*\B_0}q^*{\mathcal B}_1(-H).
\end{equation}
is the ``matrix factorization'' of the quadric $\uQ_{f(x)}=\pi^{-1}(f(x))$.
Therefore, $\Phi(L_x)$ is the cokernel of the matrix factorization map, namely the ideal sheaf $\I_{\ul_x,\uQ_{f(x)}}$.
This is well explained in \cite{Addington:PhD}.

\subsection*{Step 2: the right mutation}
Set $\Phi'(L_x):=\cat{R}_{\O_{\uY}(-h)}\Phi(L_x)$.
We want to show that $\Phi'(L_x)$ sits in the following distinguished triangle
\begin{equation}\label{eqn:phi2Lx}
	 \O_{\uY}(-h)^{\oplus 2}[1]\to \Phi'(L_x) \to \I_{\ul_x,\uQ_{f(x)}}.
\end{equation}

By Step 1, we have
\begin{equation*}
	\Phi'(L_x)=\cat{R}_{\O_{\uY}(-h)}(\I_{\ul,\uQ})= {\rm cone} \left( \I_{\ul,\uQ} \stackrel{\rm ev^\vee}\lra \RHom(\I_{\ul,\uQ}, \O_{\uY}(-h))^\vee \otimes
	\O_{\uY}(-h) \right) [-1].
\end{equation*}
Now, we observe that
\[
 \Ext^p(\I_{\ul,\uQ}, \O_{\uY}(-h))\cong H^{4-p}({\uY},\I_{\ul,\uQ}(-2H))\cong H^{4-p}(\uQ,\I_{\ul,\uQ}(-2)),
\]
where the first isomorphism follows from Serre duality and the fact that $K_{\uY}=-2H-h$, while for the second one, we use that $\O_{\uPP}(H)$ is a relative ample line bundle.

Since $\ul$ is a line in the quadric $\uQ$, we have
\[
H^{4-p}(\uQ,\I_{\ul,\uQ}(-2))={\begin{cases}
 \CC^2 &\text{if }p=2\\
 0 &\text{otherwise.}
 \end{cases}}
\]
Hence we get the distinguished triangle \eqref{eqn:phi2Lx}.

\subsection*{Step 3: the left mutation}
Set $\Phi''(L_x):=\cat{L}_{\O_{\uY}(h-H)}\cat{R}_{\O_{\uY}(-h)}\Phi(L_x)$.
We want to show that $\Phi''(L_x)[-1]$ is isomorphic to a sheaf sitting in the following non-split short exact sequence
\begin{equation}\label{eqn:phi3Lx}
	0\longrightarrow \O_{\uY}(-h)^{\oplus 2}\longrightarrow \Phi''(L_x)[-1]\longrightarrow \uK_x\longrightarrow 0,
\end{equation}
where $\uK_x$ is defined as the kernel of the evaluation map $\O_{\uY}(h-H)^{\oplus 2} \stackrel{\rm ev}\longrightarrow \I_{\ul_x,\uQ_{f(x)}}$.

We have by definition
\begin{equation*}
	\Phi''(L_x)= {\rm cone} \left( \RHom(\O_{\uY}(h-H),\Phi'(L_x)) \otimes
	\O_{\uY}(h-H)\stackrel{\rm ev}\lra \Phi'(L_x) \right).
\end{equation*}

By Step 2, we need to compute $\Ext^p(\O_{\uY}(h-H),\O_{\uY}(-h)^{\oplus 2})$ and $\Ext^p(\O_{\uY}(h-H),\I_{\ul,\uQ})$.
On the one hand, we have the following natural isomorphisms
\[
\begin{split}
\Ext^p(\O_{\uY}(h-H),\O_{\uY}(-h)^{\oplus 2})=
&\cong H^{p}({\uY},\O_{\uY}(-2h+H))^{\oplus 2}\\
&\cong H^{4-p}(\uY,\O_{\uY}(h-3H))^{\oplus 2}\\
&\cong H^{4-p}(\uY,\O_{\uY}(-D)\otimes \O_{\uY}(-2H) )^{\oplus 2}\\
&\cong H^{4-p}(Y,\I_{P,Y}(-2H))^{\oplus 2}=0,
\end{split}
\]
where the first isomorphisms follows from Serre duality and of the fact that $K_{\uY}=-2H+h$.
For the third one we use that $D$ is the exceptional divisor of $\sigma$.
%Hence $\O_{\uY}(h-H)$ and $\O_{\uY}(-h)$ are orthogonal.
On the other hand,
\begin{align*}
	\Ext^p(\O_{\uY}(h-H),\I_{\ul,\uQ})=H^{p}({\tilde
	Y},\I_{\ul,\uQ}(H))=\begin{cases}
	 \CC^2 &p=0\\ 0&\text{otherwise.}
	 \end{cases}
\end{align*}
Therefore, $\Phi''(L_x)={\rm cone} \left( \O_{\uY}(h-H)^{\oplus 2}\stackrel{\rm ev}\lra \Phi'(L_x) \right)$.

Taking cohomology and using the results in Step 2, we get
\begin{equation*}
	 0\to \O_{\uY}(-h)^{\oplus 2} \to \H^{-1}(\Phi''(L_x)) \to \O_{\tilde
	Y}(h-H)^{\oplus
	2} \stackrel{\rm ev}\to \I_{\ul,\uQ}\to \H^{0}(\Phi''(L_x))\to 0
\end{equation*}
Note that the evaluation map $\O_{\uY}(h-H)^{\oplus 2} \stackrel{\rm ev}\to \I_{\ul,\uQ}$ is surjective and non-split.
Hence $\H^{0}(\Phi''(L_x))= 0$ and $\Phi''(L_x)$ sits in the non-split triangle
\begin{equation*}
	 \O_{\uY}(-h)^{\oplus 2}[1]\to \Phi''(L_x) \to \uK_x[1].
\end{equation*}

\subsection*{Step 4: the blow-up}
Finally we prove that $M_x=\Xi^{-1}(L_x)[-1]$ is isomorphic to a sheaf sitting in the non-split short exact sequence
\begin{equation}\label{eqn:laM}
	 0\to \O_{Y}(-H)^{\oplus 2}\to M_x \to K_x \to 0,
\end{equation}
where $K_x$ is defined as the kernel of the evaluation map $\I_{P,Y}^{\oplus 2} \stackrel{\rm ev}\longrightarrow \I_{l_x,Q_{f(x)}}$.
Here $P\subset Y$ is the plane contained in $Y$ and $l_x$ and $Q_{f(x)}$ are the images of $\ul_x$ and $\uQ_{f(x)}$.

Indeed, we know that $\Phi''(L_x)$ is an element on $\sigma^*\cat{T}_Y$.
Hence, by the Projection Formula, $M_x$ is a sheaf.
We only need to study $\sigma_*\uK_x$ because $\sigma_*\O_{\uY}(-h)\cong\O_Y(-H)$.
The fact that $M_x\in\coh(Y)$ implies that $\sigma_*\uK_x$ is also a sheaf that we denote by $K_x$.
Since $\res{\sigma}{\uQ}$ is an isomorphism, $\sigma_*\I_{\ul_x,\uQ_{f(x)}}\cong\I_{l_x,Q_{f(x)}}$.
Since $D=H-h$ is the exceptional divisor of $\sigma$, we have $\sigma_*\O_{\uY}(h-H)\cong \I_{P,Y}$.

\medskip

For later use, we give two different descriptions of the sheaf $K_x$.
Given the quadric $Q_{f(x)}$ and a line $l_x$ in it, we denote by $l'_x$ any line in the second ruling not containing $l_x$.
When $Q_{f(x)}$ is a cone, we set $l_x'=l_x$.

\begin{lem}\label{lem:laK}
	The sheaf $K_x$ sits in the following (non-split) short exact sequences
	\begin{gather}
	0\longrightarrow \I_{P\cup Q_{f(x)},Y}^{\oplus 2}\longrightarrow K_x\longrightarrow \I_{l'_x,Q_{f(x)}}(-H)\longrightarrow 0, \label{eqn:laK1}\\
	0\longrightarrow \I_{P\cup Q_{f(x)}, Y}\longrightarrow K_x\longrightarrow \I_{P\cup l'_x,Y}\longrightarrow 0.
	\label{eqn:laK2}
	\end{gather}
\end{lem}

\begin{proof}
	Recall that $K_x\cong\sigma_*\uK_x$, where $\uK_x:=\ker (\O_{\uY}(h-H)^{\oplus 2} \stackrel{\rm ev}\to \I_{\ul,\uQ})$.
	Denoting by $i\colon\uQ\hookrightarrow\uY$ the closed embedding, it is not difficult to see that the morphism $\O_{\uY}(h-H)^{\oplus 2}\to \I_{\ul,\uQ}$ factors through $i_*i^*\O_{\uY}(h-H)^{\oplus 2}\cong\O_{\uQ}(-H)^{\oplus 2}$ as the composition
	$\O_{\uY}(h-H)^{\oplus 2} \to \O_{\uQ}(-H)^{\oplus 2} \to \I_{\ul,\uQ}$.

	Thus we have the following commutative diagram
	\begin{equation}\label{eqn:diagbig1}
		\xymatrix{
		&0\ar[d]&0\ar[d]\\
		&\I_{\uQ,\uY}(h-H)^{\oplus 2}\ar[d]\ar@{=}[r]&\I_{\uQ,\uY}(h-H)^{\oplus 2}\ar[d]\\
		0\ar[r]&\uK_x\ar[r]\ar[d]&\O_{\uY}(h-H)^{\oplus 2}\ar[r]^(.65){\rm ev}\ar[d]&\I_{\ul,\uQ}\ar[r]\ar@{=}[d]&0\\
		0\ar[r]&\I_{\ul',\uQ}(-H)\ar[r]\ar[d]&\O_{\uQ}(-H)^{\oplus 2}\ar[r]\ar[d]&\I_{\ul,\uQ}\ar[r]&0,\\
		&0&0}
	\end{equation}
	because the kernel of the surjective morphism $\O_{\uQ}(-H)^{\oplus 2} \to \I_{\ul,\uQ}$ described above is precisely $\I_{\ul',\uQ}(-H)$.
	Here $\ul'$ is a line in the ruling opposite to the one containing $\ul$ if $\uQ$ is non-degenerate while, when $\uQ$ is a cone, we take $\ul'=\ul$.

	Applying the functor $\sigma_*$, the previous diagram becomes
	\begin{equation*}
		\xymatrix{
		&0\ar[d]&0\ar[d]\\
		&\I_{P\cup Q, Y}^{\oplus 2}\ar[d]\ar@{=}[r]&\I_{P\cup Q, Y}^{\oplus 2}\ar[d]\\
		0\ar[r]&K_x\ar[r]\ar[d]&\I_{P, Y}^{\oplus 2}\ar[r]\ar[d]&\I_{l,Q}\ar[r]\ar@{=}[d]&0\\
		0\ar[r]&\I_{l',Q}(-H)\ar[r]\ar[d]&\O_{Q}(-H)^{\oplus 2}\ar[r]\ar[d]&\I_{l,Q}\ar[r]&0.\\
		&0&0}
	\end{equation*}

	Considering the first column of the previous diagram we can also see $K_x$ as the following extension
	\begin{equation*}
		\xymatrix{
		&0\ar[d]&0\ar[d]\\
		&\I_{P\cup Q, Y}\ar[d]\ar@{=}[r]&\I_{P\cup Q, Y}\ar[d]\\
		0\ar[r]&\I_{P\cup Q, Y}^{\oplus 2}\ar[r]\ar[d]&K_x\ar[r]\ar[d]&\I_{l',Q}(-H)\ar[r]\ar@{=}[d]&0\\
		0\ar[r]&\I_{P\cup Q, Y}\ar[r]\ar[d]&T\ar[r]\ar[d]&\I_{l',Q}(-H)\ar[r]&0,\\
		&0&0}
	\end{equation*}
	where $T$ corresponds to an element in $\Ext^1(\I_{l',Q}(-H),\I_{P\cup Q, Y})$.
	To compute this class observe that, since $P\cup Q$ is a complete intersection in $Y$, we have the
	short exact sequence
	\begin{equation*}
		\xymatrix{
		0\ar[r]&\O_Y(-2H)\ar[r]^{\alpha}&\O_Y(-H)^{\oplus 2}\ar[r]&\I_{P\cup Q, Y}\ar[r]& 0,
		}
	\end{equation*}
	where $\alpha$ is the multiplication by the equations of two hyperplanes in $Y$.

	Thus, applying the exact functor $\mathbf{R}\Hom(\I_{l',Q}(-H),-)$ to such a sequence and taking cohomology, we get
	\begin{equation*}
		0\to \Ext^1(\I_{l',Q}(-H),\I_{P\cup Q, Y})\to\CC^2\stackrel{\varphi}\to \CC^{12}\to \Ext^2(\I_{l',Q}(-H),\I_{P\cup Q,
		Y})\to 0,
	\end{equation*}
	as $\dim\Ext^2(\I_{l',Q}(-H),\O_Y(-2H))=2$ and $\dim\Ext^2(\I_{l',Q}(-H),\O_Y(-H))=6$.

	Of course, $\varphi=\alpha[2]\circ(-)$ and so it vanishes.
	By Serre duality, $\Ext^1(\I_{l',Q}(-H),\I_{P\cup Q,Y})\cong H^0(Q,\I_{l,Q}(H))\cong \CC^2$.
	Thus $\Ext^1(\I_{l',Q}(-H),\I_{P\cup Q,Y})$ gets identified to the $1$-dimensional linear system of lines $l'$ in $Q$ and $T=\I_{P\cup l',Y}$.

	The fact that the exact sequence $$0\longrightarrow\uK_x\longrightarrow\O_{\uY}(h-H)^{\oplus 2}\stackrel{\rm ev}\longrightarrow\I_{\ul,\uQ}\longrightarrow 0$$ does not split implies that \eqref{eqn:laK1} and \eqref{eqn:laK2} are non-split as well.
\end{proof}

From Lemma \ref{lem:laK}, we can deduce in a standard way the following Chern characters, by using Grothendieck--Riemann--Roch:
\begin{align*}
	\ch(\I_{P\cup Q_{f(x)},Y})&=(1,0,-H^2,3l,-\tfrac{7}{4}),\\
	\ch(\I_{P\cup l_x',Y})&=(1,0,-P,-l,\tfrac{7}{4}),\\
	\ch(\O_Y(-H)^{\oplus 2})&=(2,-2H,H^2,-l,\tfrac{1}{4}),\\
	\ch(K_x)&=(2,0,-P-H^2,2l,0),\\
	\ch(M_x)&=(4,-2H,-P,l,\tfrac{1}{4}).
\end{align*}
This completes the proof of Proposition \ref{prop:family}.

\section{A family of stable ACM vector bundles}\label{sec:family2}

We are now ready to conclude the proof of Theorem \ref{thm:main4folds}. More precisely, the aim of this section is to prove the following proposition.

\begin{prop}\label{prop:family2}
	For all $x\in S_{\mathrm{reg}}$, the sheaf $M_x$ is a Gieseker stable ACM bundle on $Y$.
\end{prop}

The proof is divided in two steps.

\subsection*{Step 1: ACM bundle}
In order to prove that $M_x$ is an ACM bundle, we want to apply Lemma \ref{lem:viceversa}.
Hence, we need:
\[
\begin{split}
& H^1(Y,M_x(H))=0\\
& H^1(Y,M_x(-3H))=H^2(Y,M_x(-3H))=H^3(Y,M_x(-3H))=0.
\end{split}
\]
These vanishing will be proved in Lemma \ref{lem:van1} and Lemma \ref{lem:lastcohom}.
\begin{lem}\label{lem:van1}
	For $x\in S_{\mathrm{reg}}$, we have
	$H^1(Y,M_x(mH))= H^2(Y,M_x(mH))= 0$, for all $m\in \ZZ$.
	%\begin{align*}
	%H^0(Y,M_x(m))&= 0 &\text{for all } m\leq 0\\
	% H^1(Y,M_x(mH))&= 0 &\text{for all } m\\
	% H^2(Y,M_x(mH))&= 0 &\text{for all } m\geq -3\\
	%H^3(Y,M_x(m))&= 0 & \text{for all }m\geq -2\\
	%H^4(Y,M_x(m))&= 0 & \text{for all }m\geq -2.
	%\end{align*}
\end{lem}

\begin{proof}
	From \eqref{eqn:laM}, by tensoring by $\O_Y(mH)$ and taking cohomology, we get
	\[
	 H^p(Y,M_x(mH))=H^p(Y,K_x(mH))\qquad \text{for } p=1,2 \text{ and for all }m\in\ZZ.
	\]

	To compute the cohomology of $K_x(mH)$, we want to use \eqref{eqn:laK1}.
	As above, we take the tensor product $\O_Y(mH)$ and then cohomology. Thus we get
	\[
	 \begin{split}
	 H^1(Y,\I_{P\cup Q}(mH))^{\oplus 2}\to H^1(Y,K_x(mH))\to H^1(Y,\I_{l',Q}((m-1)H))\to H^2(Y,\I_{P\cup Q}(mH))^{\oplus 2}
	%&\to H^2(Y,K_x(mH))\to H^2(Y,\I_{l',Q}((m-1)H))
	 \end{split}
	\]

	By using the Koszul resolution of $\I_{P\cup Q}$, we have the vanishings
	\begin{equation*}
	H^1(Y,\I_{P\cup Q}(mH))= H^2(Y,\I_{P\cup Q}(mH))= 0.
	%\qquad H^4(Y,\I_{P\cup Q}(m))\cong H^0(Y,\O_Y(-3-m))^\vee
	\end{equation*}

	From the matrix factorization presentation \eqref{eqn:matrixf}, we have $H^1(Y,\I_{l',Q}((m-1)H))=0$ for all $m\in \ZZ$.
	Therefore, $H^1(Y,M_x(mH))=0$.

To conclude the proof of the lemma we only need to show that $H^2(Y,K_x(-mH))=0$.
	Since $K_x$ is defined as the kernel of the evaluation map $\I_{P,Y}^{\oplus 2} \stackrel{\rm ev}\longrightarrow \I_{l,Q}$, we have an inclusion $H^2(Y,K_x(-mH))\hookrightarrow H^2(Y,\I_{P,Y}(-mH))^{\oplus 2}$ and $H^2(Y,\I_{P,Y}(-mH))=0$.
\end{proof}

In order to prove the crucial vanishing, $ H^3(Y,M_x(-3H))=0$, we set
\begin{equation}\label{eqn:laN}
	 N:=\cat{L}_{\O_{\uY}(2H)} (\cat{R}_{\O_{\uY}(-3H)}(\O_{\uY}(H-h))\otimes
	\O_{\uY}(2H+h)).
\end{equation}

\begin{lem}\label{lem:tech1}
	The object $N\in \Db(\uY)$ lies in the following distinguished triangle
	\begin{equation*}
		\O_{\uY}(-H+h)[3]\longrightarrow N\longrightarrow\res{\sigma^*\Omega_{\PP^5}(3H)}{Y}[1].
	\end{equation*}
\end{lem}

\begin{proof}
	First we consider the right mutation of $\O_{\uY}(H-h)$ with respect to $\O_{\uY}(-3H)$, that is
	\begin{equation*}
		 {\rm cone} \left( \O_{\uY}(H-h) \stackrel{\rm ev^\vee}\lra \RHom(\O_{\uY}(H-h), \O_{\uY}(-3H))^\vee
		\otimes
		\O_{\uY}(-3H) \right) [-1].
	\end{equation*}
	Using Serre duality we have that $\Ext^p(\O_{\uY}(H-h), \O_{\uY}(-3H))^\vee$ is $1$-dimensional if $p=4$ and it is trivial otherwise.
	Hence we have the distinguished triangle
	\begin{equation*}
		\O_{\uY}(H-h) \stackrel{\rm ev^\vee}\lra \O_{\uY}(-3H)[4]\to \cat{R}_{\O_{\uY}(-3H)}(\O_{\uY}(H-h))[1]
	\end{equation*}
	which, after tensorization by $\O_{\uY}(2H+h)$ and shift, becomes
	\begin{equation}\label{eqn:laN0}
		\O_{\uY}(-H+h)[3] \to \cat{R}_{\O_{\uY}(-3H)}(\O_{\uY}(H-h))\otimes \O_{\uY}(2H+h) \to \O_{\uY}(3H).
	\end{equation}

	Now we want to compute the left mutation of the middle term in \eqref{eqn:laN0} with respect to $\O_{\uY}(2H)$.
	To this end, we compute the left mutations of the first and third term in the same distinguished triangle.
	Now, an easy computation shows that $\cat{L}_{\O_{\uY}(2H)}(\O_{\uY}(-H+h)[3])\cong\O_{\uY}(-H+h)[3]$.

	On the other hand, the vector space $\Ext^p(\O_{\uY}(2H),\O_{\uY}(3H))$ is $6$-dimensional if $p=0$ and trivial otherwise.
	Thus we get a distinguished triangle
	\begin{equation*}
		\O_{\uY}(2H)^{\oplus 6} \stackrel{\rm ev}\longrightarrow \O_{\uY}(3H)\longrightarrow \cat{L}_{\O_{\uY}(2H)}(\O_{\uY}(3H)).
	\end{equation*}
	Hence $\cat{L}_{\O_{\uY}(2H)}(\O_{\uY}(3H))[-1]\cong\res{\sigma^*\Omega_{\PP^5}(3H)}{Y}$ and putting all together we get the desired conclusion.
\end{proof}

Finally we can prove the following.

\begin{lem}\label{lem:lastcohom}
	For $x\in S_{\mathrm{reg}}$, we have $H^3(Y,M_x(-3H))=0$.
\end{lem}

\begin{proof}
	Using adjunction, we get
	\begin{equation*}
	\begin{split}
	H^3(Y,M_x(-3H))&=\Hom^2_{\Db(Y)}(\O_Y(3H),\Xi^{-1}(L_x))\\
	&=\Hom_{\Db(\uY)}^{2}(\O_{\uY}(3H), \cat{L}_{\O_{\uY}(h-H)}\cat{R}_{\O_{\uY}(-h)}\Phi(L_x))\\
	% &\cong\Ext^{p-1}(\O_{\uY}(3H), i_{\O_Y(h-H)^\perp}i^*_{\O_Y(h-H)^\perp}(\Phi'(L_x) ))\\
	% &\cong\Ext^{5-p}( i_{\O_Y(h-H)^\perp}i^*_{\O_Y(h-H)^\perp}(\Phi'(L_x) ),\O_{\uY}(H-h))^\vee.
	% \end{split}
	% \end{equation}
	%
	% Denote by $E$ the latter (dual) vector space.
	% Proceeding further, we can prove that the latter (dual) vector space is naturally isomorphic to the following ones
	% \begin{equation}\label{eqn:stup2}
	% \begin{split}
	% E&\cong\Ext^{5-p}(\Phi'(L_x)
	% 	,i_{\O_Y(h-H)^\perp}i_{\O_Y(h-H)^\perp}^!\O_{\uY}(H-h))^\vee\\
	% &\cong\Ext^{5-p}(\Phi'(L_x),i_{{}^\perp\O_Y(-3H)}i_{{}^\perp\O_Y(-3H)}^!\O_{\uY}(H-h))^\vee\\
	% &\cong\Ext^{5-p}(i_{{}^\perp\O_Y(-h)}i_{{}^\perp\O_Y(-h)}^!\Phi(L_x),\cat{R}_{\O_{\uY}(-3H)}\O_{\uY}(H-h))^\vee\\
	% &\cong\Ext^{p-1}(\cat{R}_{\O_{\uY}(-3H)}(\O_{\uY}(H-h))\otimes\O_{\uY}(2H+h),i_{{}^\perp\O_Y(-h)}i_{{}^\perp\O_Y(-h)}^!\Phi(L_x))\\
	% &\cong\Ext^{p-1}(\cat{R}_{\O_{\uY}(-3H)}(\O_{\uY}(H-h))\otimes\O_{\uY}(2H+h),i_{\O_Y(2H)^\perp}i_{\O_Y(2H)^\perp}^!\Phi(L_x))\\
	% &\cong\Ext^{p-1}(\cat{L}_{\O_{\uY}(2H)} (\cat{R}_{\O_{\uY}(-3H)}(\O_{\uY}(H-h))\otimes\O_{\uY}(2H+h)),\Phi(L_x))\\
	&=\Hom_{\Db(\PP^2,\B_0)}^{2}(\Psi(\cat{L}_{\O_{\uY}(2H)} (\cat{R}_{\O_{\uY}(-3H)}(\O_{\uY}(H-h))\otimes\O_{\uY}(2H+h))),L_x)\\
	&=\Hom_{\Db(\PP^2,\B_0)}^{2}(\Psi(N),L_x).
	\end{split}
	\end{equation*}
	where $N$ is defined in \eqref{eqn:laN}.
	By \cite[Lem 2.3]{MS}, we have $\Psi(\O_{\uY}(-H+h))=0$.
	Therefore, by Lemma \ref{lem:tech1},
	\begin{equation*}
	 \Psi(N)\cong\Psi(\res{\sigma^*\Omega_{\PP^5}(3H)}{\uY})[1],
	\end{equation*}
	and thus,
	\begin{equation*}
	\begin{split}
	H^3(Y,M_x(-3H))&\cong\Hom_{\Db(\PP^2,\B_0)}^{1}(\Psi(\res{\sigma^*\Omega_{\PP^5}(3H)}{\uY}, L_x)\\
	&=\Ext_{\uY}^{1}(\res{\sigma^*\Omega_{\PP^5}(3H)}{\uY}, \I_{\ul,\uQ})\\
	&=\Ext^{1}_{Y}(\res{\Omega_{\PP^5}(3H)}{Y}, \I_{l, Q})\\
	&= H^{1}(Q,\res{\T_{\PP^5}(-3H)}{Q}\otimes \I_{l, Q})\\
	&\cong H^1(Q,\res{\T_{\PP^3}(-3H)}{Q}\otimes \I_{l, Q})\oplus H^1(Q, \I_{l, Q}(-2H))^{\oplus 2}.
	\end{split}
	\end{equation*}
	The last isomorphism follows from the fact that $Q\subset \PP^3$ and $\T_{\PP^5}\otimes \O_{\PP^3}\cong \T_{\PP^3}\oplus \O_{\PP^3}(H)^{\oplus 2}$.
	% On the other hand, the third one is \eqref{eqn:Kuz} while the fourth one follows from the fact that $\sigma^*$ is fully faithful (see \cite{Orlov}).

	From the matrix factorization presentation \eqref{eqn:matrixf}, we have $H^1(Q,\I_{l,Q}(-2H))=0$.
	Hence, to conclude the proof of the lemma we only need to show $H^1(Q,\T_{\PP^3}(-3H)\otimes \I_{l,Q})=0$.

	The Euler sequence on $\PP^3$ restricted to the quadric and twisted by $\I_{l,Q}(-3H)$ becomes
	\begin{equation*}
	0\longrightarrow \I_{l,Q}(-3H)\longrightarrow \I_{l,Q}(-2H)^{\oplus 4}\longrightarrow \T_{\PP^3}(-3H)\otimes \I_{l,Q}\longrightarrow 0.
	\end{equation*}
	Taking cohomology we see that $H^1(Q,\T_{\PP^3}(-3H)\otimes \I_{l,Q})$ is the kernel of the morphism
	\[
	 H^2(Q,\I_{l,Q}(-3H))\xrightarrow{\alpha} H^2(Q,\I_{l,Q}(-2H)^{\oplus 4}).
	\]
	Hence, we need to show that $\alpha$ is injective.
	Consider the short exact sequence
	\[
	0\longrightarrow\I_{l,Q}\longrightarrow\O_{Q}(H)^{\oplus 2}\longrightarrow\I_{l',Q}(H)\longrightarrow 0.
	\]
	%appearing in the bottom line of \eqref{eqn:diagbig1}.
	By taking cohomology and using that $H^1(Q,\I_{l',Q}(-2H))=H^1(Q,\I_{l',Q}(-H))=0$, we have that $\alpha$ sits in the following commutative diagram
	\begin{equation*}
	 \xymatrix{
	H^2(Q,\O_{Q}(-3H)^{\oplus 2})\ar[r]^\beta& H^2(Q,\O_{Q}(-2H)^{\oplus 8})\\
	H^2(Q,\I_{l,Q}(-3H))\ar[r]^\alpha\ar@{^{(}->}[u]& H^2(Q,\I_{l,Q}(-2H)^{\oplus 4}\ar@{^{(}->}[u]).
	}
	\end{equation*}
	Thus, to prove that $\alpha$ is injective is enough to show that $\beta$ is such.
	By construction, $\beta=H^2(f)^{\oplus 2}$ where $H^2(f)$ is the morphism induced on cohomology by the map $f$ sitting in the following Koszul exact sequence on $Q$
	\begin{equation*}
	0\longrightarrow\O_{Q}(-3H)\stackrel{f}\longrightarrow \O_{Q}(-2H)^{\oplus 4}\longrightarrow\O_{Q}(-H)^{\oplus 6}\longrightarrow\O_{Q}^{\oplus 4}\stackrel{g}\longrightarrow \O_{Q}(H)\longrightarrow 0.
	\end{equation*}
	The cokernel of $f$ is $\T_{\PP^3}(-3H)\otimes\O_Q$ and the kernel of $g$ is $\Omega_{\PP^3}(H)\otimes \O_Q$.
	Chasing through the associated long exact sequence in cohomology, we get $H^1(Q,\T_{\PP^3}(-3H)\otimes \O_Q)=H^0(Q,\Omega_{\PP^3}(H)\otimes \O_Q)=0$, since $H^0(Q,\O_{Q}^{\oplus 4})\to H^0(Q,\O_{Q}(H))$ is a base change of the evaluation map.
	Hence, $H^2(f)$ is injective.
\end{proof}

\subsection*{Step 2: Gieseker stability}
Now we finish the proof of Proposition \ref{prop:family2} by showing that, again for all $x\in S_{\mathrm{reg}}$, the ACM bundle $M_x$ is stable.
% To this end, recall the non-split short exact sequences \eqref{eqn:laK2} and \eqref{eqn:laM} respectively
% \begin{gather*}
% 0\longrightarrow \I_{P\cup Q, Y}\longrightarrow K_x\longrightarrow \I_{P\cup l',Y}\longrightarrow 0\\
% 0\longrightarrow \O_{Y}(-H)^{\oplus 2}\longrightarrow M_x\longrightarrow K_x\longrightarrow 0.
% \end{gather*}
Note that $\mu(K_x)=0$, $\mu(M_x)=-\frac{1}{2}$, and by \eqref{eqn:laK2}, $K_x$ is $\mu$-semistable.

Suppose $M_x$ is not stable.
Since $M_x$ is a vector bundle, there exists $F$ a semistable destabilizing reflexive sheaf and a sheaf $G$ sitting in a short exact sequence
\begin{equation*}
	0\longrightarrow F\longrightarrow M_x\longrightarrow G\longrightarrow 0
\end{equation*}
with $\mu(F)\geq -\tfrac{1}{2}$.
Now $\rk(F)=1,2,3$ and the three cases need to be analysed separately.

\medskip

\noindent{\em Case A: $\rk(F)=1$.}
Then $\mathrm{c}_1(F)\geq 0$ and $F$ is a line bundle.
Moreover, we have a commutative diagram
\begin{equation*}
	 \xymatrix{F\ar@{^{(}->}[d]\ar@{^{(}->}^\phi[rd]\\M_x\ar@{>>}[r]& K_x.}
\end{equation*}
Since $F$ is torsion free, the composition $\phi$ can only vanish or be an injection.
If $\phi$ is trivial, then it factors through $\O_Y(-H)^{\oplus 2}$ which is semistable with $\mu=-1$ (recall that $\O_Y(H)$ generates $\Pic(Y)$).
Thus we get a contradiction.
So assume that $\phi$ is injective and consider the following commutative diagram
\begin{equation*}
	 \xymatrix{&F\ar@{^{(}->}[d]^-{\phi}\ar^\varphi[rd]\\\I_{P\cup Q,Y}\ar@{^{(}->}[r]&K_x\ar@{>>}[r]& \I_{P\cup l',Y}.}
\end{equation*}
Again, $F$ cannot inject neither in $\I_{P\cup Q,Y}$ nor in $\I_{P\cup l',Y}$, and we get a contradiction.

\medskip

\noindent{\em Case B: $\rk(F)=2$.}
In this case $c_1(F)\geq -1$ and we have a commutative diagram
\begin{equation}\label{eqn:stab1}
	 \xymatrix{F_1\ar@{^{(}->}[r]\ar@{^{(}->}[d]&F\ar@{^{(}->}[d]\ar@{>>}[r]&F_2\ar@{^{(}->}[d]\\ \O_Y(-H)^{\oplus 2}\ar@{^{(}->}[r]&M_x\ar@{>>}[r]& K_x.}
\end{equation}
At this point we have to analyse some additional cases.

\smallskip

\noindent{\em Case B.1: $\rk(F_1)=2$.}
Since $K_x$ is torsion free, $F_2=0$.
On the other hand, $O_Y(-H)^{\oplus 2}$ is semistable, so $c_1(F_1)\leq -2$.
Then $c_1(F)\leq -2$, so $\mu(F)\leq -1$ and $F$ does not destabilize $M_x$.

\smallskip

\noindent{\em Case B.2: $\rk(F_1)=\rk(F_2)=1$.}
In that situation, $c_1(F_1)=-1$, so $F_1\cong \O_Y(-H)$ and $c_1(F_2)=0$.
On the other hand, since $F_2\hookrightarrow K_x$, by \eqref{eqn:laK1},
we have $F_2\hookrightarrow \I_{P\cup Q}$.
Thus \eqref{eqn:stab1} can be rewritten as
\begin{equation*}
 \xymatrix{\O_Y(-H)\ar@{^{(}->}[r]\ar@{^{(}->}[d]&F\ar@{^{(}->}[d]\ar@{>>}[r]&F_2\ar@{^{(}->}[d]\\
\O_Y(-H)^{\oplus 2}\ar@{>>}[d]\ar@{^{(}->}[r]&M_x\ar@{>>}[d]\ar@{>>}[r]&
K_x\ar@{>>}[d]\\
     \O_Y(-H)\ar@{^{(}->}[r]& G\ar@{>>}[r] & G_2.}
\end{equation*}
In that case,
$\ch_2(F)=\ch_2(\O_Y(-H))+\ch_2(F_2)\leq\tfrac{H^2}{2}-(P+Q)=-\tfrac{H^2}{2}$.
Since $\tfrac{\ch_2(M_x)\cdot H^2}{\rk (M_x)}=-\tfrac{P\cdot H^2}{4}=
-\tfrac{H^4}{12} > -\tfrac{H^4}{4}\geq \tfrac{\ch_2(F)\cdot H^2}{\rk(F)}$, $F$
does not destabilize $M_x$.

\smallskip

\noindent{\em Case B.3: $\rk(F_2)=2$.}
In this case, $F\cong F_2$. If $c_1(F)\geq 0$, then  $F\cong K_x$.
Thus \eqref{eqn:laM} splits, which gives a contradiction.

If $c_1(F)= -1$, then $F\hookrightarrow K_x$ and, by \eqref{eqn:laK1}, we have $F\hookrightarrow \I_{P\cup Q}^{\oplus 2}$.
Hence $F$ is the extension of two ideal sheaves.
Since we have assumed that $F$ is semistable, we have that
\begin{equation*}
	0\to \I_{Z_1}(-H)\to F\to \I_{Z_2}\to 0,
\end{equation*}
where $\codim Z_1$ and $\codim Z_2$ are greater or equal than $2$.
Moreover, $Z_1$ is possibly empty and $P\cup Q \subseteq Z_2$.
Thus, $\ch_2(F)=\ch_2(\I_{Z_1}(-H))+\ch_2(\I_{Z_2})=\frac{H^2}{2}-Z_1-Z_2\leq \frac{H^2}{2}-P\cup Q$. Hence, the same computation as at the end of Case B.2 shows that $F$ does not destabilize $M_x$.

\medskip

\noindent{\em Case C: $\rk(F)=3$.}
Now $c_1(F)\geq -1$ and we can consider again a diagram as \eqref{eqn:stab1}.
We can distinguish two possibilities.

\smallskip

\noindent{\em Case C.1: $\rk(F_2)=2$.}
Since $F$ is semistable, $-\tfrac{1}{3}\leq \mu(F)\leq\mu(F_2)\leq
\mu(K_x)=0$.
As $2\mu(F_2)$ is an integer, $\mu(F_2)=c_1(F_2)=0$, $F_1\cong \O_Y(-H)$
and $c_1(F)=-1$.
Hence we rewrite again \eqref{eqn:stab1} as
\begin{equation*}
\xymatrix{\O_Y(-H)\ar@{^{(}->}[r]\ar@{^{(}->}[d]&F\ar@{^{(}->}[d]\ar@{>>}[r]&F_2\ar@{^{(}->}[d]\\
\O_Y(-H)^{\oplus 2}\ar@{>>}[d]\ar@{^{(}->}[r]&M_x\ar@{>>}[d]\ar@{>>}[r]&
K_x\ar@{>>}[d]\\
     \O_Y(-H)\ar@{^{(}->}[r]& G\ar@{>>}[r] & T.}
\end{equation*}
Since $G$ is a torsion-free sheaf of rank $1$ and $c_1(G)=-1$, if
$\O_Y(-H)\not \cong G$, then $G\cong \I_Z(-H)$ with $Z\neq \emptyset$
with $\codim Z\geq 2$.
Hence we get a contradiction since $\Hom(\O_Y(-H),\I_Z(-H))=0$.
Thus, $G\cong \O_Y(-H)$ contradicting the fact that $M_x$ is non-split.

\smallskip

\noindent{\em Case C.2: $\rk(F_2)=1$.}
Since $O_Y(-H)^{\oplus 2}$ is semistable, $c_1(F_1)\leq -2$.
On the other hand, also $K_x$ is semistable, so $c_1(F_2)\leq 0$.
Then $c_1(F)\leq -2$ and $\mu(F)\leq-\tfrac{2}{3}$.
Therefore, $F$ does not destabilize $M_x$.

\medskip

This completes the proof of Proposition \ref{prop:family2}. We are now ready to prove our main result.

\medskip

\noindent\emph{Proof of Theorem \ref{thm:main4folds}.} Consider the irreducible component $\mathfrak{M}$ of the moduli space of Gieseker stable sheaves on $Y$ with Chern character $(4,-2H,-P,l,\frac{1}{4})$ and containing the sheaves $M_x$, for $x\in S_\mathrm{reg}$.

Now observe that $M_x\in\mathbf{T}_Y$, for all $x\in S_\mathrm{reg}$, and the Serre functor of $\mathbf{T}_Y$ is the shift by $2$. Hence, by Serre duality and stability,
\[
\Hom(M_x,M_x)\cong\Ext^2(M_x,M_x)\cong\CC\qquad\Ext^1(M_x,M_x)\cong\CC^2,
\]
for all $x\in S_\mathrm{reg}$. This means that $\mathfrak{M}$ is generically smooth of dimension $2$.

By the above discussion, the functor
\[
F:=\Xi^{-1}\circ f_*(-\otimes E_\alpha^\vee)\colon\Db(S,\mathcal{A}_0)\to\Db(Y)
\]
yields an injection $S_\mathrm{reg}\hookrightarrow\mathfrak{M}$ by sending the skyscraper sheaf $\CC(x)$ to $M_x$. Moreover, such a morphism induces an isomorphism between the tangent spaces
\[
T_xS=\Ext^2(\CC(x),\CC(x))\stackrel{\sim}{\longrightarrow}\Ext^1(M_x,M_x)=T_{M_x}\mathfrak{M}.
\]
Since being right orthogonal to the three objects $\mathcal{O}_Y$, $\mathcal{O}_Y(H)$ and $\mathcal{O}_Y(2H)$ is an open condition, $F$ induces and isomorphism between $S_\mathrm{reg}$ and an open subset of $\mathfrak{M}$. If $S$ is smooth, then this gives an isomorphism $S\cong\mathfrak{M}$.

\subsection{Universal family}
In this section we assume that $S$ is smooth.
Then the above discussion can be summarized by saying that there exists a twisted universal family $\M\in \coh(S\times Y,p_1^*\alpha)$ such that the Fourier--Mukai functor
		\begin{equation*}
		\Phi_{\M}:\Db(S,\alpha)\to\Db(Y)
		\end{equation*}
is fully faithful and it factors through $\mathbf{T}_Y$.
Recall that $\Phi_{\M}(-):=(p_Y)_*(\M\otimes p_S^*(-))$.

As the Kuznetsov's functor providing the full embedding of $\Db(S,\alpha)$ into $\Db(Y)$ is a composition of a Fourier--Mukai functor and mutations, finding $\M$ amounts to finding the Fourier--Mukai kernel of their composition.
For this, consider
\begin{equation*}\mathcal{S}:=(\sigma\times \id)_*\left( p_{\uY}^*\E'\otimes_{p^*_{\uY}\pi^*\B_0} (\pi\times \id)^* (f\times \id
)_*(p_1^*E_\alpha^\vee\otimes \O_{\Delta_S})\right)
\end{equation*}
where $p_{\uY}:\uY\times S\to \uY$ is the natural projection, $p_1:S\times S\to S$ is the projection on the first factor, $\Delta_S\subset S\times S$ is the diagonal, and $\E'$ is defined in \eqref{eqn:defE'}.
From \eqref{eqn:Kuz} we have
\begin{align*}
	 \res{\mathcal{S}}{Y\times \set{x}}&\cong \sigma_*\res{\left( p_{\uY}^*\E'\otimes_{p^*_{\uY}\pi^*\B_0} (\pi\times \id)^* (f\times \id
	)_*(p_1^*E_\alpha^\vee\otimes \O_{\Delta_S})\right)}{\uY\times \set{x}}\\
	&\cong \sigma_*\left( \E'\otimes_{\pi^*\B_0} \pi^* \res{\left( (f\times \id
	)_*(p_1^*E_\alpha^\vee\otimes \O_{\Delta_S})\right)}{\PP^2\times \set{x}}\right)\\
	&\cong \sigma_*\left( \E'\otimes_{\pi^*\B_0} \pi^* f_*\left(E_\alpha^\vee\otimes \res{\O_{\Delta_S}}{S\times \set{x}}\right)\right)\\
	&\cong \sigma_*\left( \E'\otimes_{\pi^*\B_0} \pi^* f_*\left(E_\alpha^\vee\otimes \CC(x)\right)\right)\\
	&\cong \sigma_*\Phi(L_x)\\
	&\cong \I_{l_x,Q_{f(x)}}.
\end{align*}

Then, the universal family $\M$ over $Y\times S$ such that $\res{\M}{Y\times \set{x}}\cong M_x$ can be described as
\begin{equation*}
	\M:= (\sigma\times \id)_*\circ\cat{L}_{p_{\uY}^*\O_{\uY}(h-H)} \circ\cat{R}_{p_{\uY}^*\O_{\uY}(-h)}\circ(\sigma\times \id)^*\mathcal{S}[-1],
\end{equation*}
where $\cat{L}_{p_{\uY}^*\O_{\uY}(h-H)}$ and $\cat{R}_{p_{\uY}^*\O_{\uY}(-h)}$ denote the corresponding left and right mutations.

The fact that $\M\in\Db(S\times Y,p_1^*\alpha)$ is actually a locally free sheaf follows from the fact that $\res{\M}{{x}\times Y}\cong M_x$ is locally free, for all $x\in S$.
This was observed above.

%%%%%%%%%%%%%%%%%%%%%%%%

\medskip

{\small\noindent{\bf Acknowledgements.} Parts of this paper were written while the three authors were visiting the University of Bonn and the University of Barcelona whose warm hospitality is gratefully acknowledged.
We also thank the American Institute of Mathematics for sponsoring the workshop ``Brauer groups and obstruction problems: moduli spaces and arithmetic'' held February 25 to March 1, 2013, in Palo Alto, California, where parts of this paper were discussed.
It is a pleasure to thank Nick Addington, Asher Auel, Marcello Bernardara, Robin Hartshorne, Daniel Huybrechts, Nathan Ilten, Sukhendu Mehrotra, Scott Nollet, Nicolas Perrin, Antonio Rapagnetta,  and Pawel Sosna for very useful conversations and comments.
}

%%%%%%%%%%%%%%%%%%%%%%%%

\end{document}